\documentclass[11pt]{article}%
\usepackage{amssymb}
\usepackage{eurosym}
\usepackage{amsfonts}
\usepackage{amsmath}
\usepackage{graphicx}
\usepackage{epstopdf}%
\providecommand{\U}[1]{\protect\rule{.1in}{.1in}}
\newtheorem{theorem}{Theorem}
\newtheorem{acknowledgement}[theorem]{Acknowledgement}

\newtheorem{definition}[theorem]{Definition}

\newtheorem{lemma}[theorem]{Lemma}

\newtheorem{proposition}[theorem]{Proposition}
\newtheorem{remark}[theorem]{Remark}

\newenvironment{proof}[1][Proof]{\noindent\textbf{#1.} }{\ \rule{0.5em}{0.5em}}
\begin{document}

\title{Trajectories of a quadratic differential related to a quasi-exactly solvable
sextic oscillator}
\author{M.J.Atia, W.Karrou, M.Chouikhi, and F.Thabet\\University of Gab\`{e}s, Tunisia}
\maketitle

\begin{abstract}
In this paper, we discuss the existence of solution (as Cauchy transform of a
signed measure) of a particular algebraic quadratic equation of the form ;
$z\mathcal{C}^{2}(z)-(z^{2}+\frac{\gamma}{2}z)\mathcal{C}(z)+(z+\frac{\delta
}{4})=0.$ This problem remains to describe the critical graph of a related
meromorphic quadratic differential; in particular, we discuss the existence of
finite critical trajectories of this quadratic differential.

\end{abstract}


AMS classification scheme numbers: 13F60, 34M60

Quadratic differentials.

\section{Introduction}

This paper is a continuation of a collection of works about applications of
the theory of quadratic differentials in quantum theory. In fact, quadratic
differentials have provided an important tool in the asymptotic study of some
solutions of algebraic equations. In quantum theory, trajectories of some
quadratic differentials have crucial role in the WKB analysis.

We consider the eigenvalue problem
\begin{equation}
-y^{\prime\prime}+V_{s,m}(x)y=\lambda y,\text{ } \label{sch eq}%
\end{equation}
where the potential $V_{s,J}(x)$ is given by :%
\[
V_{s,m}(x)=\frac{(4s-1)(4s-3)}{x^{2}}+(x^{6}-(4s+4m-2)x^{2}).
\]
This problem was studied by A.Turbiner \cite{ref01},\cite{ref3}. For
$s\in\left]  \frac{1}{4},\frac{3}{4}\right[  $, there is an attractive
centrifugal term. For $s\in\left]  -\infty,\frac{1}{4}\right[  \cup$ $\left]
\frac{3}{4},+\infty\right[  ,$ the centrifugal term is repulsive. For
$s=\frac{1}{4}$ or $\frac{3}{4},$ the centrifugal core term disappears leaving
a sextic oscillator
\[
V_{p,m}(x)=\text{ }x^{6}-(4m+p)x^{2};\text{ }p\in\left\{  -1,1\right\}
\]
We will study the case $s=\frac{1}{4}$ (so $p=-1$), with a perturbed
potential
\[
V_{-1,m+1}(x)=V_{m}(x)=\text{ }x^{6}+\gamma m^{1/2}x^{4}+(\frac{\gamma^{2}%
m}{4}-(4m+3))x^{2},\text{ }%
\]
with boundray condition $y(\pm\infty)=0.$This problem is exactly solvable:
that means that for every $\gamma$ there exist $m+1$ eigenfunctions of the
form
\[
\phi(x)=q_{m}(x)e^{-\frac{x^{4}}{4}-\frac{\gamma m^{1/2}}{4}x^{2}},
\]
where $q_{m}$ is an even-parity polynomial of degree $2m$, corresponding to
$m+1$ eigenvalues $\lambda_{m};$ see \cite{ref2}. Note that in the case $p=1$
we have the same form of eigenfunctions with an odd-parity polynomials; for
more details see again \cite{ref2}.

Substitiuting in the Schr\"{o}dinger equation (\ref{sch eq}), we find that
\[
-q_{m}^{\prime\prime}(x)+(2x^{3}+\gamma m^{\frac{1}{2}}x)q_{m}^{\prime
}(x)-(4mx^{2}-\frac{\gamma m^{1/2}}{2})q_{m}(x)=\lambda_{m}q_{m}(x).
\]
The differential operator%
\[
K=-\frac{d^{2}}{dx^{2}}+(2x^{3}+\gamma m^{1/2}x)\frac{d}{dx}-(4mx^{2}%
-\frac{\gamma m^{1/2}}{2})
\]
preserves the $(m+1)$-dimentional linear space of all even polynomials of
degree $\leq2m$. Using $z=x^{2},$ we find that
\begin{equation}
-4zq_{m}^{\prime\prime}(z)+(4z^{2}+2\gamma m^{1/2}z-2)q_{m}^{\prime
}(z)-(4mz-\frac{\gamma m^{1/2}}{2})q_{m}(z)=\lambda_{m}q_{m}(z). \label{eq2}%
\end{equation}

Our main goal is the study of the asymptotic root-counting measure
$\vartheta_{m}$ of an appropriate re-scaled polynomial sequence $\left\{
Q_{m}(z)=q_{m}(m^{\varepsilon}z)\right\}  .$ Let $\nu_{m}$ be the normalized
root-counting measure of the sequence $\left(  q_{m}\right)  $. The Cauchy
transform $C_{\nu_{m}}$ and $C_{\vartheta_{m}}$ of respectively $\nu_{m}$ and
$\vartheta_{m}$ satisfies the following equations :
\[
4zmC_{\nu_{m}}^{2}(z)-(4z^{2}+2\gamma m^{1/2}z-2)C_{\nu_{m}}(z)+\frac
{(4mz-\frac{\gamma m^{1/2}}{2}+\lambda_{m})}{m}+4zC_{\nu_{m}}^{\prime}(z)=0.
\]%
\[
4zm^{1-\varepsilon}C_{\vartheta_{m}}^{2}(z)-(4m^{\varepsilon}z^{2}+2\gamma
m^{1/2}z-2m^{-\varepsilon})C_{\vartheta_{m}}(z)+\frac{(4m^{1+\varepsilon
}z-\frac{\gamma m^{1/2}}{2}+\lambda_{m})}{m}+4m^{-\varepsilon}zC_{\vartheta
_{m}}^{\prime}(z)=0.
\]
For $\varepsilon=1/2,$ we get
\begin{equation}
-4zC_{\vartheta_{m}}^{2}(z)+(4z^{2}+2\gamma z-\frac{2}{m})C_{\vartheta_{m}%
}(z)-(4z-\frac{\gamma}{2m}+\frac{\lambda m}{m^{3/2}})-4z\frac{C_{\vartheta
_{m}}^{\prime}(z)}{m}=0. \label{eqdiff1}%
\end{equation}
It was shown in \cite{ref4}, that the sequence $\left(  \lambda_{m}%
/m^{4/3}\right)  $ is bounded. By the Helly selection Theorem, we may assume
that
\[
\underset{m\rightarrow\infty}{\lim}\left(  \lambda_{m}/m^{4/3}\right)
=\delta,
\]
and then, there exists a compactly-supported positive measure $\nu$ such that
\[
\underset{m\rightarrow\infty}{\lim}\vartheta_{m}=\nu,\underset{m\rightarrow
\infty}{\lim}C_{\vartheta_{m}}\mathcal{=C}_{\nu}=\mathcal{C}.
\]
Finally, taking the limits in (\ref{eqdiff1}), we obtain the algebraic
equation :%
\begin{equation}
z\mathcal{C}^{2}(z)-(z^{2}+\frac{\gamma}{2}z)\mathcal{C}(z)+(z+\frac{\delta
}{4})=0. \label{algeq}%
\end{equation}

\bigskip In this paper, we discuss the existence of solutions of equation
(\ref{algeq}) as Cauchy transform of compactly-supported signed measure.
Section \ref{connection alg}, we make the connection between this a algebraic
equation and a particular quadratic differential. In section \ref{qudar dif},
we describe the critical graph of the related quadratic differential in the
Riemann sphere $\overline{%
%TCIMACRO{\U{2102} }%
%BeginExpansion
\mathbb{C}
%EndExpansion
},$ more precisely, we discuss the number of its finite critical trajectories.

\section{\bigskip A quadratic differential \label{qudar dif}}

\bigskip Below, we describe the critical graphs of the the family of quadratic
differentials
\begin{equation}
\varpi_{q}=-\frac{q\left(  z\right)  }{z}dz^{2}, \label{qd}%
\end{equation}
where $q$ is a monic polynomial of degree $3.$ We begin our investigation by
some immediate observations from the theory of quadratic differentials. For
more details, we refer the reader to \cite{Strebel},\cite{jenkins},.

Recall that \emph{finite critical points }of a given meromorphic quadratic
differential $-Q\left(  z\right)  dz^{2}$ on the Riemann sphere $\overline{%
%TCIMACRO{\U{2102} }%
%BeginExpansion
\mathbb{C}
%EndExpansion
}$ are its zeros and simple poles; poles of order $2$ or greater then 1 called
\emph{infinite critical points. }All other points of $\overline{%
%TCIMACRO{\U{2102} }%
%BeginExpansion
\mathbb{C}
%EndExpansion
}$ are called \emph{regular points}.

\emph{Horizontal trajectories} (or just trajectories) of the quadratic
differential are the zero loci of the equation%
\[
-Q\left(  z\right)  dz^{2}>0,
\]
or equivalently%
\begin{equation}
\mathcal{\Re}\int^{z}\sqrt{Q\left(  t\right)  }\,dt=\text{\emph{const}}.
\label{eq traj}%
\end{equation}
If $z\left(  t\right)  ,t\in%
%TCIMACRO{\U{211d} }%
%BeginExpansion
\mathbb{R}
%EndExpansion
$ is a horizontal trajectory, then the function
\[
t\longmapsto\Im\int^{t}\sqrt{Q\left(  z\left(  u\right)  \right)  }z^{\prime
}\left(  u\right)  du
\]
is monotone.

The \emph{vertical} (or, \emph{orthogonal}) trajectories are obtained by
replacing $\Im$ by $\Re$ in equation (\ref{eq traj}). The horizontal and
vertical trajectories produce two pairwise orthogonal foliations of the
Riemann sphere $\overline{%
%TCIMACRO{\U{2102} }%
%BeginExpansion
\mathbb{C}
%EndExpansion
}$.

A trajectory passing through a critical point is called \emph{critical
trajectory}. In particular, if it starts and ends at a finite critical point,
it is called \emph{finite critical trajectory }or\emph{\ short trajectory},
otherwise, we call it an \emph{infinite critical trajectory}. A short
trajectory is called \emph{unbroken }if it does not pass through any finite
critical points except its two endpoints. The closure the set of finite and
infinite critical trajectories is called the \emph{critical graph}.

A necessary condition for the existence of a short trajectory connecting
finite critical points is the existence of a Jordan arc $\gamma$ connecting
them, such that
\begin{equation}
\Re\int_{\gamma}\sqrt{Q\left(  t\right)  }dt=0. \label{cond necess}%
\end{equation}
However, this condition is sufficient in general; see counter-example in
\cite{F.Thabet}.

The local structure of the trajectories is as follow :

\begin{itemize}
\item At any regular point, horizontal (resp. vertical) trajectories look
locally as simple analytic arcs passing through this point, and through every
regular point passes a uniquely determined horizontal (resp. vertical)
trajectory; these horizontal and vertical trajectories are locally orthogonal
at this point.

\item From each zero of multiplicity $r$, there emanate $r+2$ critical
trajectories spacing under equal angle $2\pi/\left(  r+2\right)  $.

\item At a simple pole, there emanates exactly one horizontal trajectory.

\item At the pole of order $r>2$, there are $r-2$ asymptotic directions
(called \emph{critical directions}) spacing under equal angle $2\pi/\left(
r-2\right)  $ and a neighborhood $\mathcal{U}$, such that each trajectory
entering $\mathcal{U}$ stays in $\mathcal{U}$ and tends to the pole in one of
the critical directions. See Figure \ref{FIG1}.
\end{itemize}

\begin{figure}[tbh]
\begin{minipage}[b]{0.28\linewidth}
\centering\includegraphics[scale=0.25]{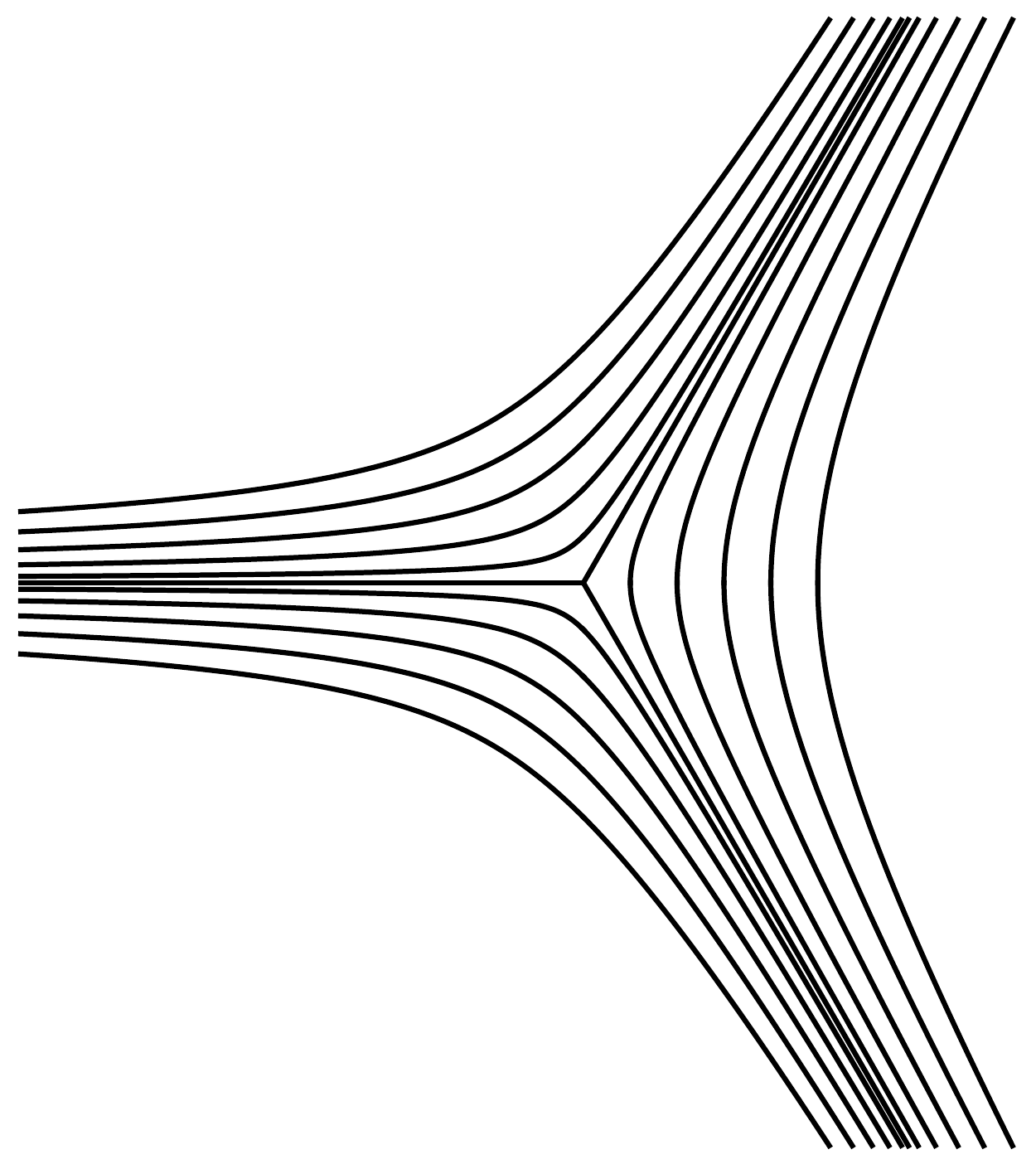}
\end{minipage}\hfill
\begin{minipage}[b]{0.28\linewidth} \includegraphics[scale=0.25]{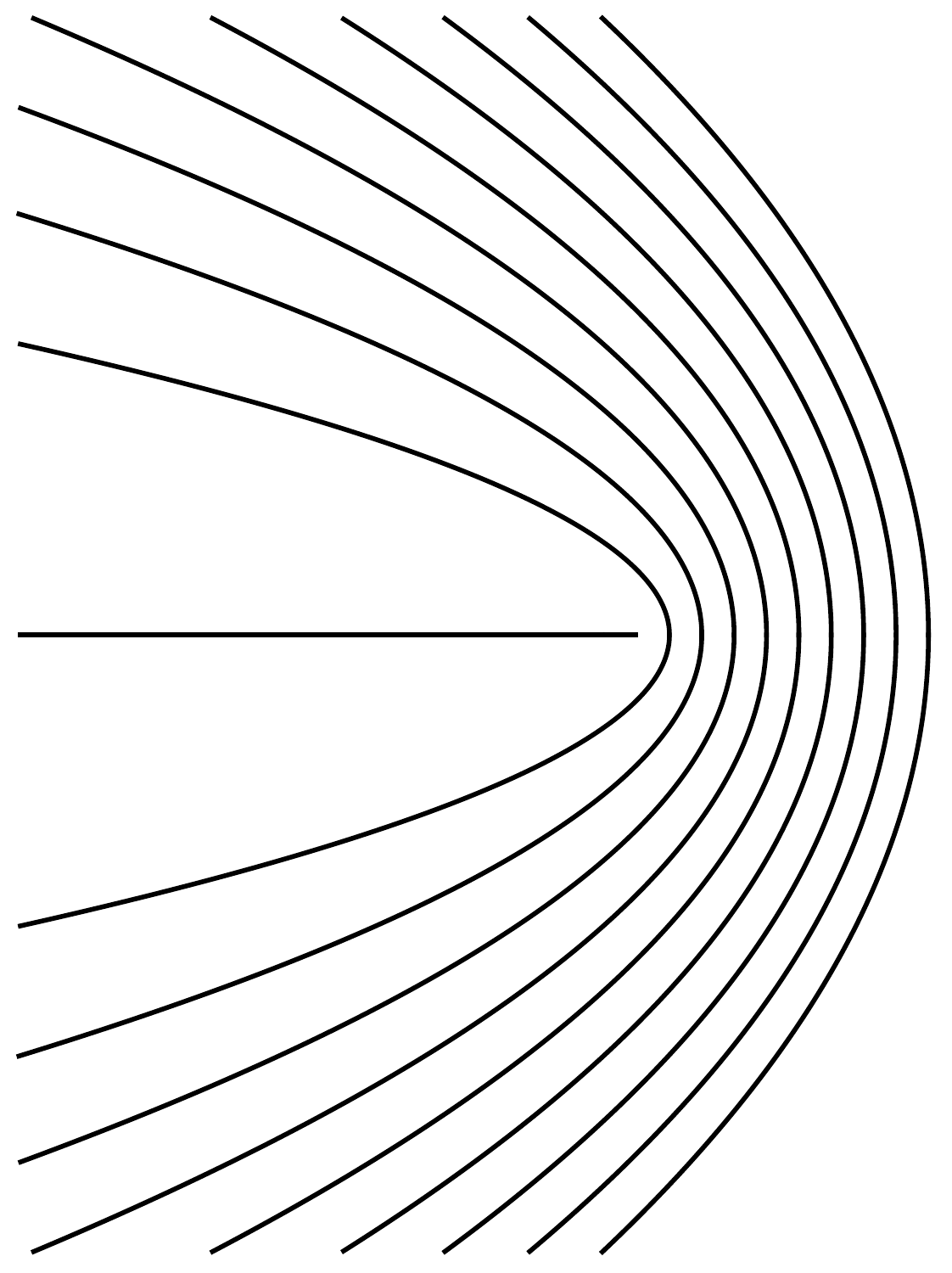}
\end{minipage}
\hfill\begin{minipage}[b]{0.28\linewidth} \includegraphics[scale=0.3]{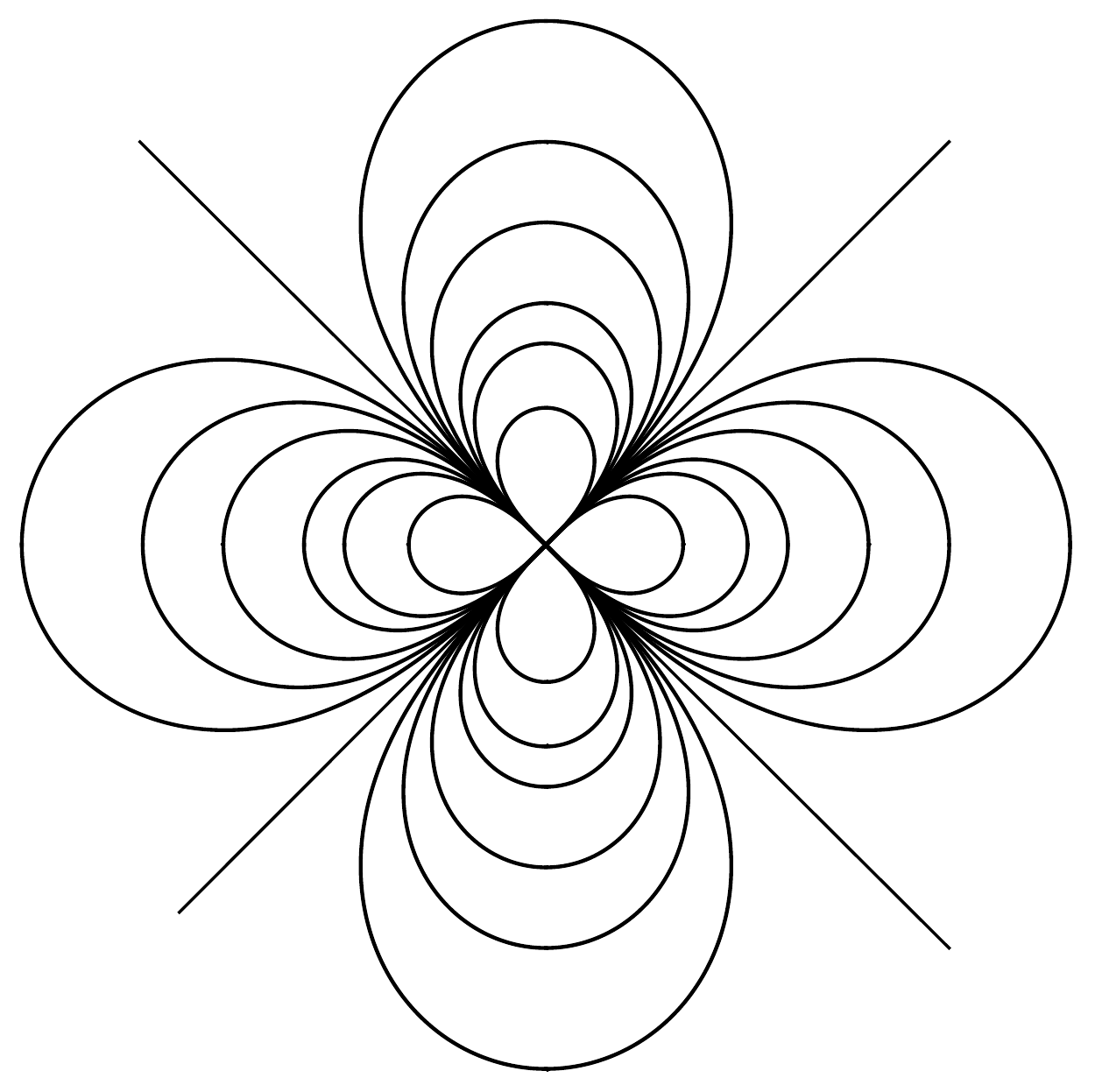}
\end{minipage}
\caption{Structure of the trajectories near a simple zero (left), a simple
pole (center), and a pole of order 6 (right).}%
\label{FIG1}%
\end{figure}

\bigskip A very helpful tool that will be used in our investigation is the
Teichm\"{u}ller lemma (see \cite[Theorem 14.1]{Strebel}).

\begin{definition}
\bigskip A domain in $\overline{%
%TCIMACRO{\U{2102} }%
%BeginExpansion
\mathbb{C}
%EndExpansion
}$ bounded only by segments of horizontal and/or vertical trajectories of
$\varpi_{q}$ (and their endpoints) is called $\varpi_{q}$-polygon.
\end{definition}

\begin{lemma}
[Teichm\H{u}ller]\label{teich lemma} Let $\Omega$ be a $\varpi_{q}$-polygon,
and let $z_{j}$ be the critical points on the boundary $\partial\Omega$ of
$\Omega,$ and let $\theta_{j}$ be the corresponding interior angles with
vertices at $z_{j},$ respectively. Then%
\begin{equation}
\sum\left(  1-\dfrac{\left(  n_{j}+2\right)  \theta_{j}}{2\pi}\right)  =2+\sum
m_{i}, \label{Teich equality}%
\end{equation}
where $n_{j}$ are the multiplicities of $z_{j}=1,$ and $m_{i}$ the
multiplicities of critical points inside $\Omega.$\bigskip
\end{lemma}

\bigskip We will focus on the case where
\[
q\left(  z\right)  =q_{a}\left(  z\right)  =\left(  z-1\right)  \left(
z-a\right)  \left(  z-\overline{a}\right)  ,
\]
with
\[
a\in%
%TCIMACRO{\U{2102} }%
%BeginExpansion
\mathbb{C}
%EndExpansion
^{+}=\left\{  a\in%
%TCIMACRO{\U{2102} }%
%BeginExpansion
\mathbb{C}
%EndExpansion
\mid\Im\left(  a\right)  >0\right\}  .
\]
We have the following immediate observations :

\begin{itemize}
\item The finite critical points of $\varpi_{q}$ are $1,a,\overline{a}$ as
simple zeros, and $-1$ as a simple pole.

\item With the parametrization $u=1/z$, we get
\[
\varpi_{q}\left(  u\right)  =\left(  -\frac{1}{u^{6}}+\mathcal{O}\left(
\frac{1}{u^{5}}\right)  \right)  du^{2},\text{ }u\longrightarrow0,
\]
thus, infinity is an infinite critical point of $\varpi_{q},$ as a pole of
order $6.$

\item Since $\infty$ is the only infinite critical point of $\varpi_{q},$ any
critical trajectory which is not finite diverges to $\infty$ following one of
the 4 directions :
\[
D_{k}=\left\{  z\in%
%TCIMACRO{\U{2102} }%
%BeginExpansion
\mathbb{C}
%EndExpansion
\mid\arg\left(  z\right)  =\left(  2k+1\right)  \frac{\pi}{4}\right\}
,k=0,1,2,3.
\]
The same thing happen to the orthogonal trajectories at $\infty$, but the
critical directions are :%
\[
D_{k}^{\perp}=\left\{  z\in%
%TCIMACRO{\U{2102} }%
%BeginExpansion
\mathbb{C}
%EndExpansion
:\arg\left(  z\right)  =\frac{k\pi}{2}\right\}  ;k=0,1,2,3.
\]
Observe that if two trajectories diverge to $\infty$ in a same direction
$D_{k},$ then there exists a neighborhood $\mathcal{V}$ of $\infty$, such that
any orthogonal trajectory traversing $D_{k}$ in $\mathcal{V},$ must traverse
these two trajectories.

\item \bigskip Since the quadratic differential $\varpi_{q}$ has two poles,
Jenkins Three-pole Theorem (see \cite[Theorem 15.2]{Strebel}) asserts that the
situation of the so-called recurrent trajectory (whose closure might be dense
in some domain in $%
%TCIMACRO{\U{2102} }%
%BeginExpansion
\mathbb{C}
%EndExpansion
$) cannot happen.
\end{itemize}

\bigskip

\begin{lemma}
\label{at infinity}Two critical trajectories of $\varpi_{q}$ emanating from
the same zero cannot diverge to $\infty$ in the same direction.
\end{lemma}

\begin{lemma}
\label{residue}For any Jordan arc $\gamma$ connecting $a$ and $\overline{a}$
in $%
%TCIMACRO{\U{2102} }%
%BeginExpansion
\mathbb{C}
%EndExpansion
\setminus\left[  0,1\right]  $ we have%
\[
\Re\int_{\gamma}\sqrt{\frac{\left(  z-1\right)  \left(  z-a\right)  \left(
z-\overline{a}\right)  }{z}}dz=0,
\]
and then, condition (\ref{cond necess}) is fulfilled.
\end{lemma}

We consider the set
\[
\Sigma=\left\{  z\in%
%TCIMACRO{\U{2102} }%
%BeginExpansion
\mathbb{C}
%EndExpansion
:\Re\int_{0}^{z}\sqrt{\frac{\left(  t-1\right)  \left(  t-z\right)  \left(
t-\overline{z}\right)  }{t}}dt=0\right\}
\]

\begin{lemma}
\label{curve}The set $\Sigma$ is symmetric with respect to the real axis, and
it is formed by $3$ Jordan arcs :

\begin{itemize}
\item the segment $\left[  0,1\right]  ;$

\item two curves $\Sigma^{\pm}$ emerging from $z=1,$ and diverging
respectively to infinity in $%
%TCIMACRO{\U{2102} }%
%BeginExpansion
\mathbb{C}
%EndExpansion
_{\pm}.$
\end{itemize}
\end{lemma}

\begin{figure}[th]
\centering\includegraphics[height=2in,width=3in]{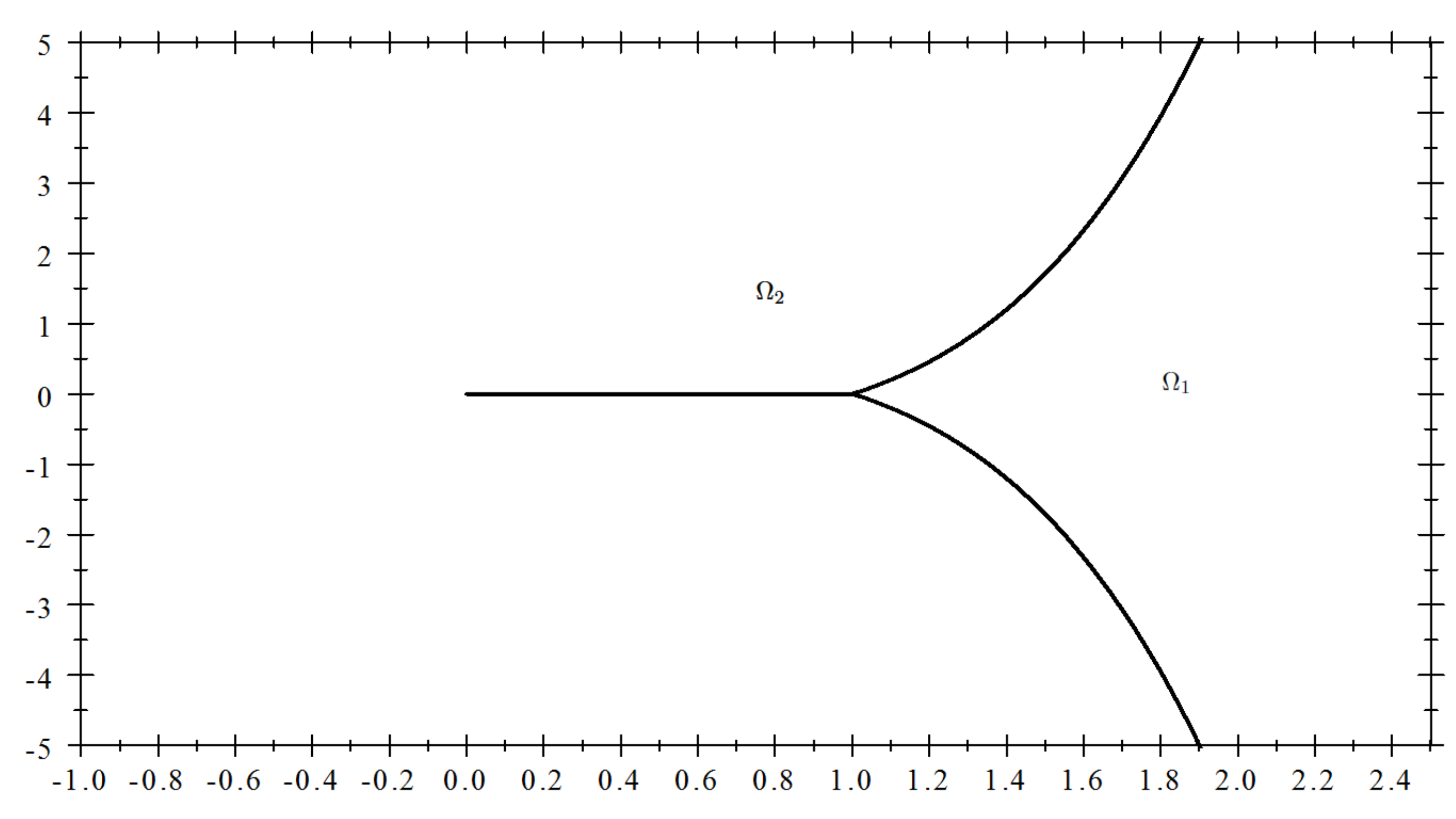}\caption{Approximate
plot of the curve $\Sigma$}%
\label{FIG2}%
\end{figure}

We give here the behavior of $\Sigma$ at $z=1$ and at the infinity:

\begin{lemma}
\label{asympt of the curve}The following results hold :
\begin{align*}
\lim\limits_{z\rightarrow\infty,z\in\Sigma\cap%
%TCIMACRO{\U{2102} }%
%BeginExpansion
\mathbb{C}
%EndExpansion
^{+}}\arg\left(  z\right)   &  =\frac{\pi}{2},\\
\lim\limits_{z\rightarrow1,z\in\Sigma\cap%
%TCIMACRO{\U{2102} }%
%BeginExpansion
\mathbb{C}
%EndExpansion
^{+}}\arg\left(  z\right)   &  =\frac{\pi}{3}.
\end{align*}

\end{lemma}

From Lemma \ref{curve}, $\Sigma$ splits $%
%TCIMACRO{\U{2102} }%
%BeginExpansion
\mathbb{C}
%EndExpansion
$ into 2 connected domains :

\begin{itemize}
\item $\Omega_{1}$ limited by $\Sigma^{\pm}$ and containing $z=2;$

\item $\Omega_{2}=%
%TCIMACRO{\U{2102} }%
%BeginExpansion
\mathbb{C}
%EndExpansion
\setminus$ $\left(  \Omega_{1}\cup\Sigma^{\pm}\cup\left[  0,1\right]  \right)
.$ See Figure \ref{FIG2}.
\end{itemize}

\begin{proposition}
\label{main}For any complex number $a,$ the quadratic differential $\varpi
_{q}$ has :

\begin{itemize}
\item two short trajectories if $a\in\Omega_{1}$: the segment $\left[
0,1\right]  $ and another one connecting $a$ and $\overline{a}$ in $\Omega
_{1};$see Figure \ref{FIG3};

\item three short trajectories if $a\in\Sigma^{\pm}$ : the segment $\left[
0,1\right]  $ and two others connecting $z=1$ with $a$ and $\overline{a};$ see
Figure \ref{FIG4};

\item two short trajectories if $a\in\Omega_{2}$ : the segment $\left[
0,1\right]  $ and another one connecting $a$ and $\overline{a}$ in $\Omega
_{2};$see Figure \ref{FIG5}.
\end{itemize}
\end{proposition}

\begin{figure}[tbh]
\begin{minipage}[b]{0.48\linewidth}
\centering\includegraphics[scale=0.5]{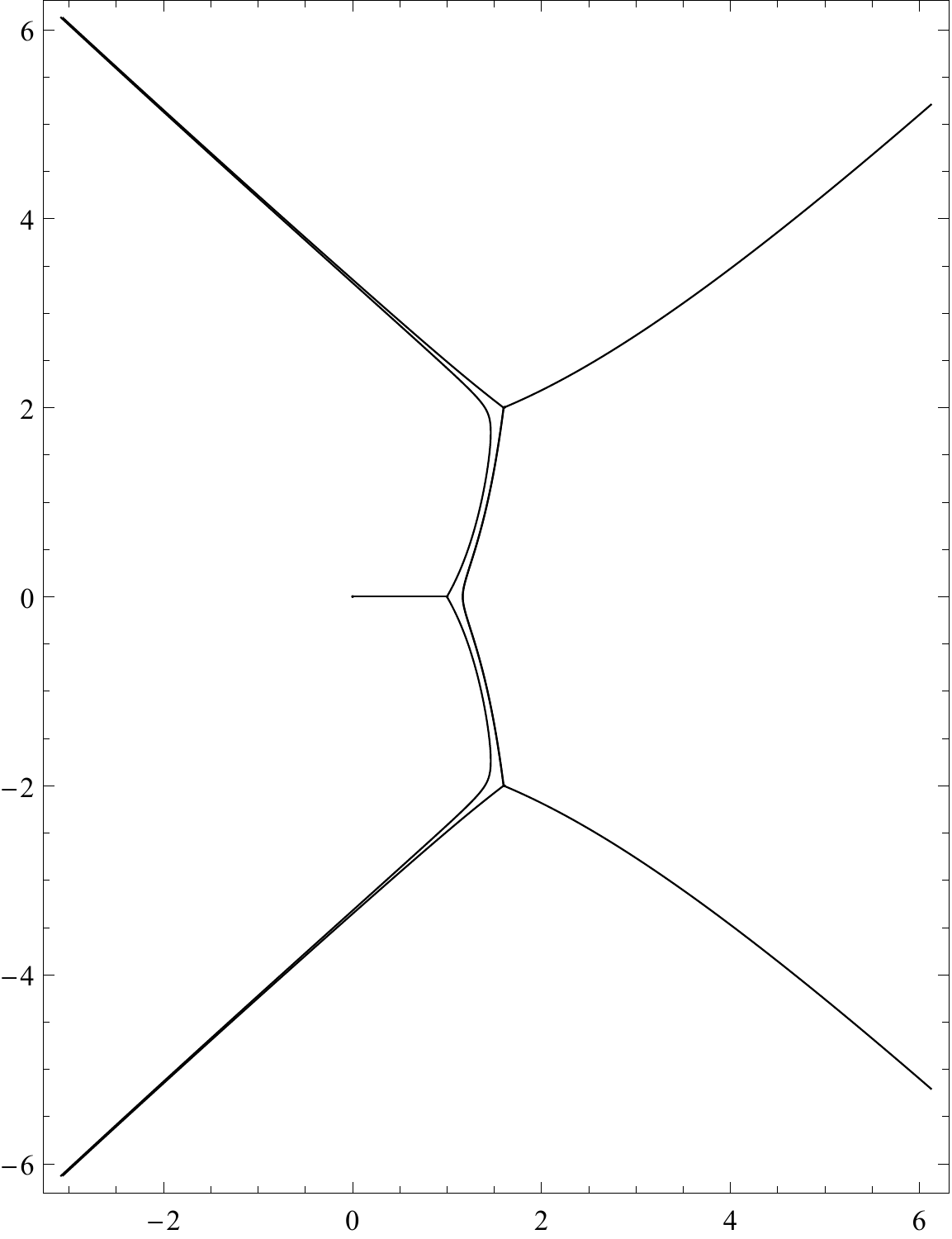}
\end{minipage}\hfill
\begin{minipage}[b]{0.48\linewidth} \includegraphics[scale=0.5]{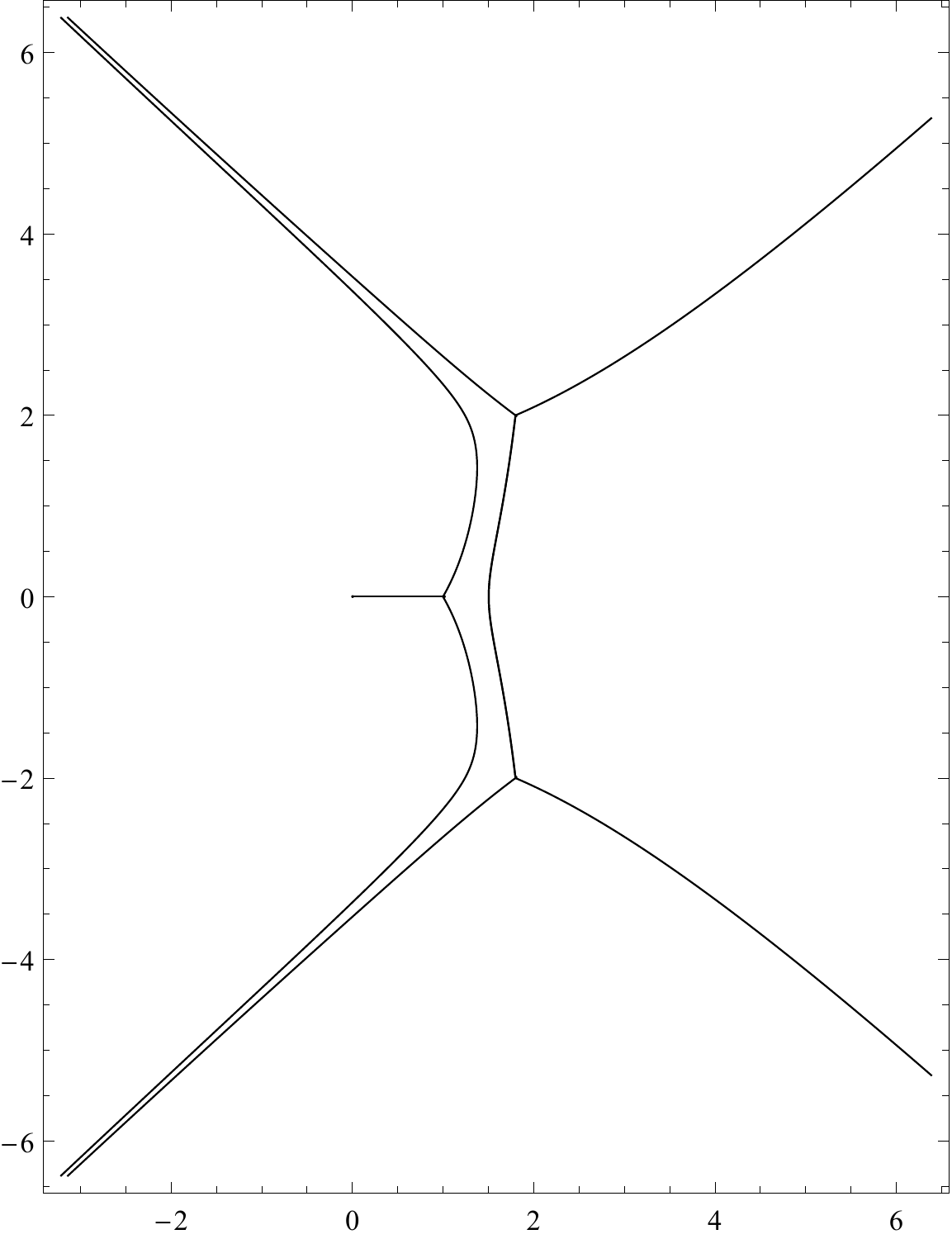}
\end{minipage}
\caption{Critical graphs when $a\in\Omega_{1},$ here $a=1.6+2i$ (left) and
$a=1.8+2i$ (right).}%
\label{FIG3}%
\end{figure}

\begin{figure}[th]
\centering\includegraphics[height=2in,width=3in]{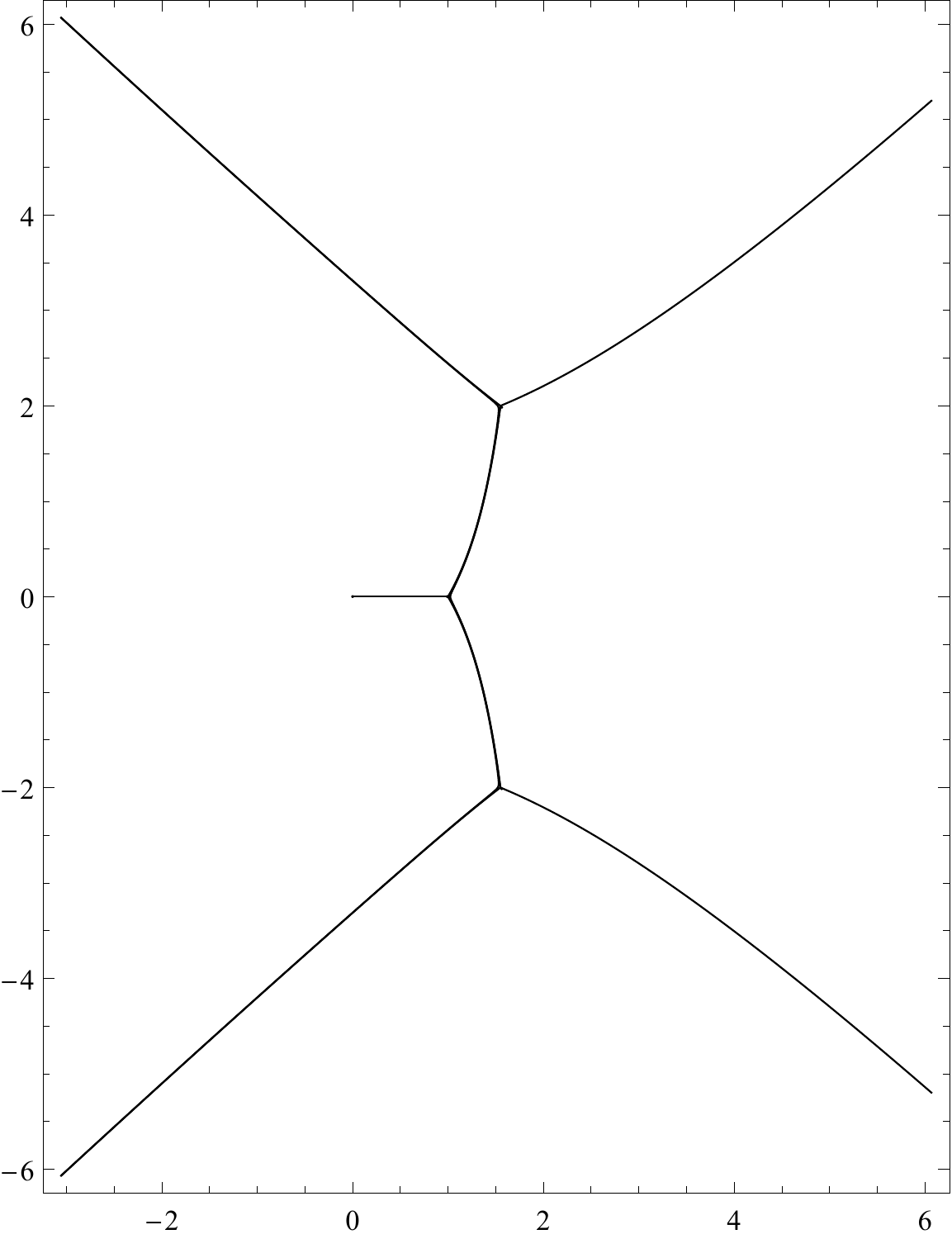}\caption{Critical graph
when $a$ $\in\Sigma,$ here $a=1.55+2i.$}%
\label{FIG4}%
\end{figure}

\begin{figure}[tbh]
\begin{minipage}[b]{0.48\linewidth}
\centering\includegraphics[scale=0.5]{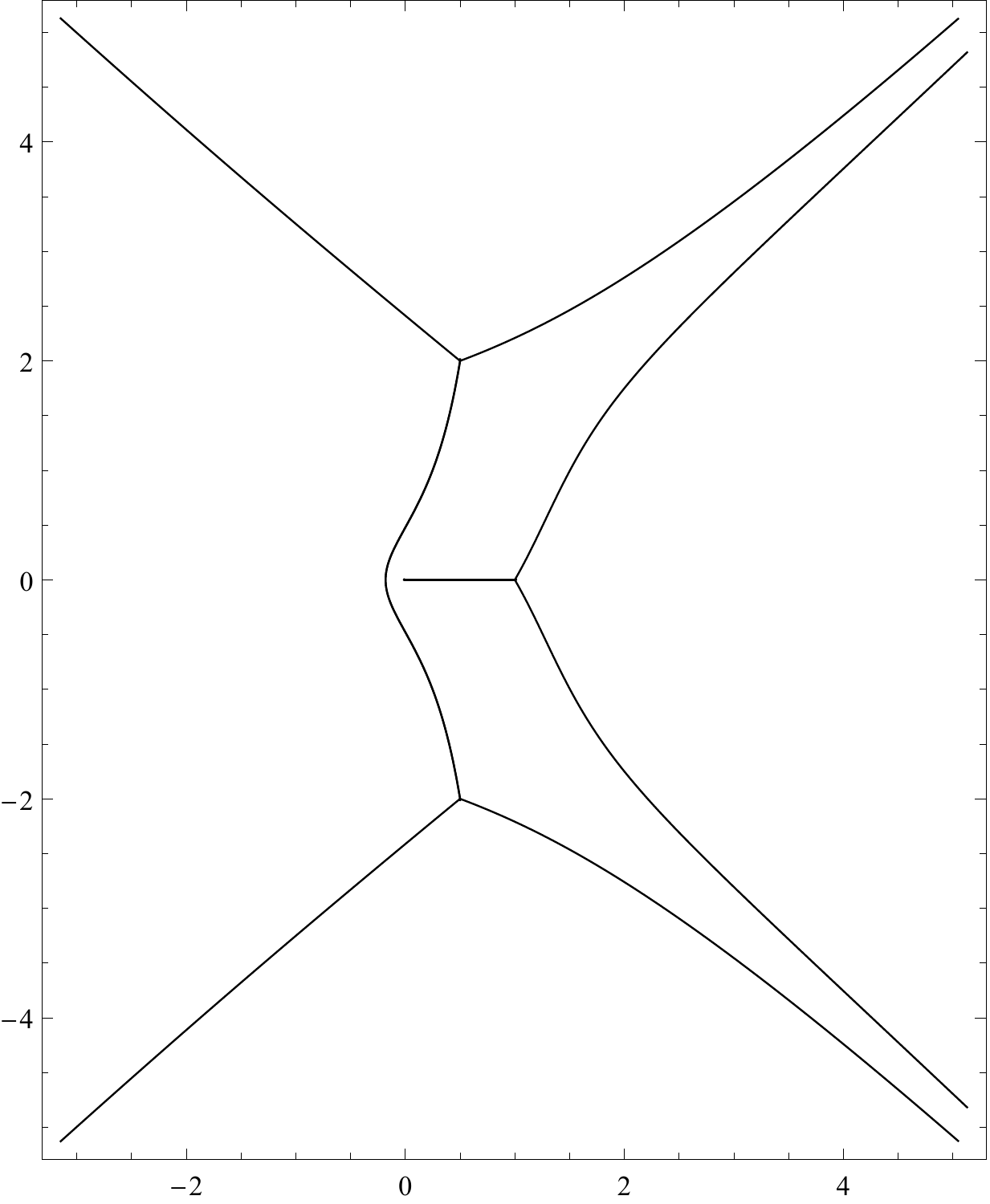}
\end{minipage}\hfill
\begin{minipage}[b]{0.48\linewidth} \includegraphics[scale=0.5]{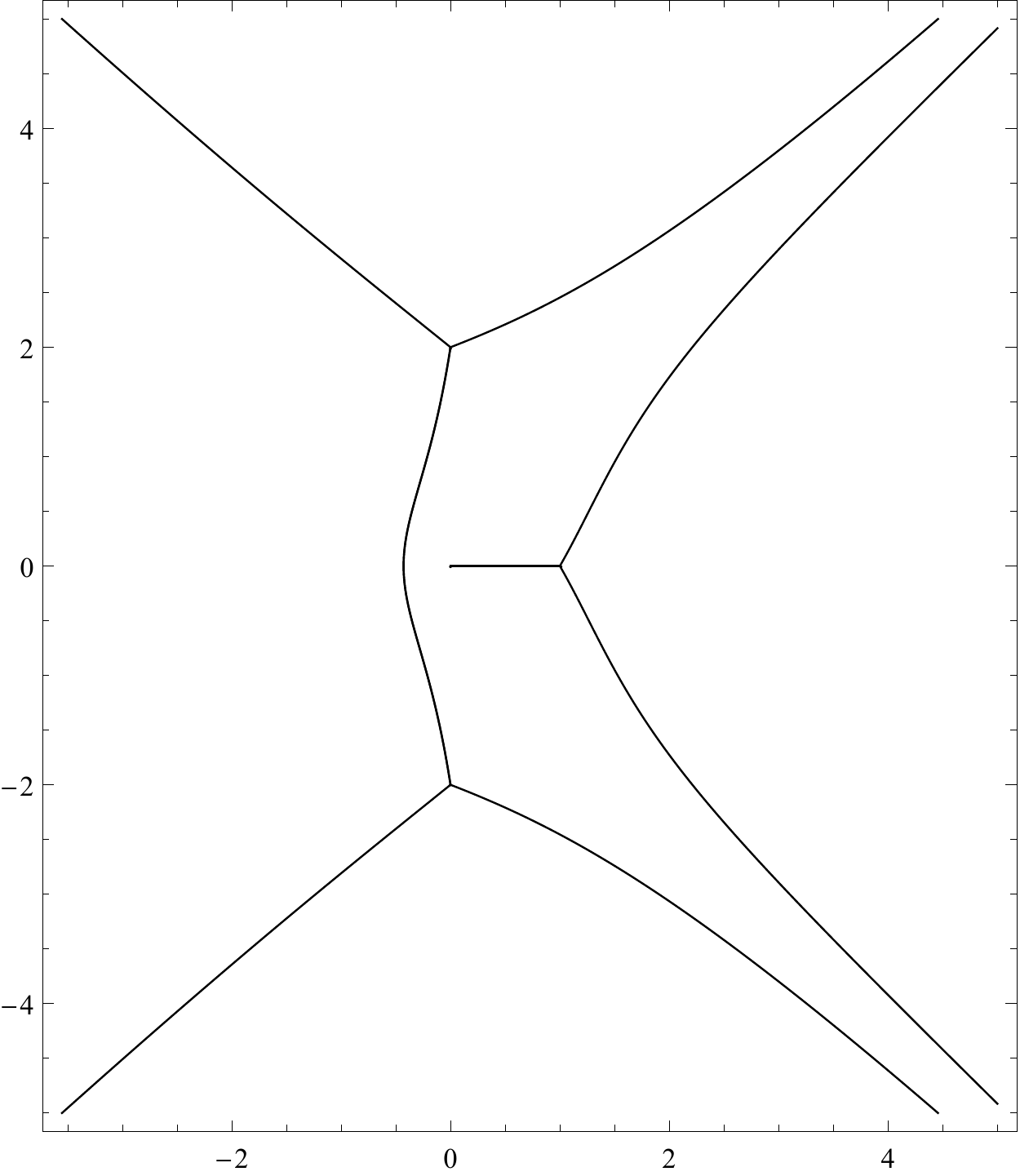}
\end{minipage}
\caption{Critical graphs when $a\in\Omega_{2},$ here $a=0.5+2i$ (left) and
$a=2i$ (right).}%
\label{FIG5}%
\end{figure}

\begin{remark}
The case when $q=\prod_{i=1}^{3}\left(  z-a_{i}\right)  \in%
%TCIMACRO{\U{211d} }%
%BeginExpansion
\mathbb{R}
%EndExpansion
\left[  X\right]  ,$ with real zeros is quite simple; if the zeros are simple
: $a_{1}<a_{2}<a_{3},$ then the segments $\left[  a_{1},a_{2}\right]  $ and
$\left[  a_{3},a_{4}\right]  $ are two short trajectories of $\varpi_{q}$. See
Figure \ref{FIG6}.
\end{remark}

\begin{figure}[tbh]
\begin{minipage}[b]{0.48\linewidth}
\centering\includegraphics[scale=0.5]{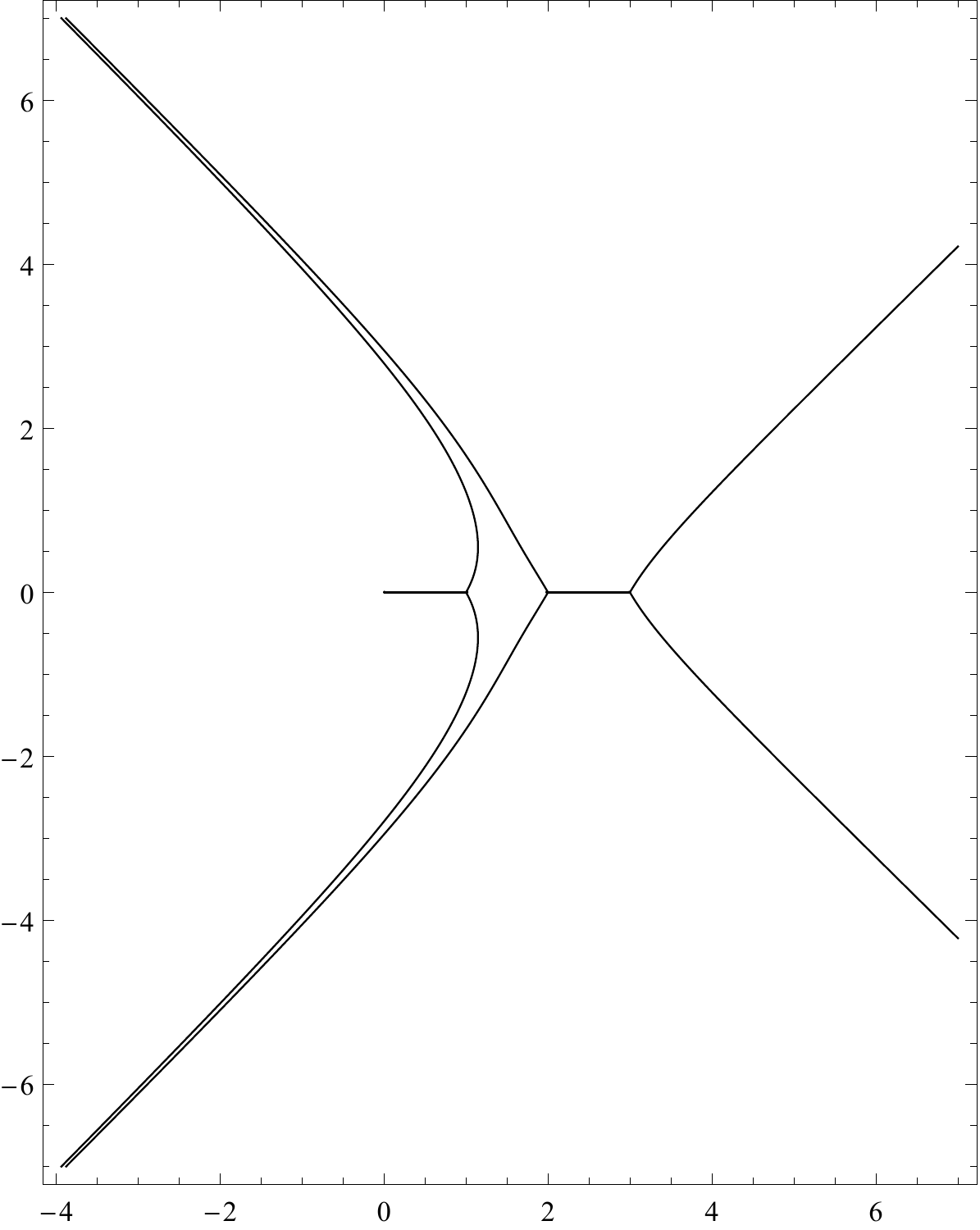}
\end{minipage}\hfill
\begin{minipage}[b]{0.48\linewidth} \includegraphics[scale=0.5]{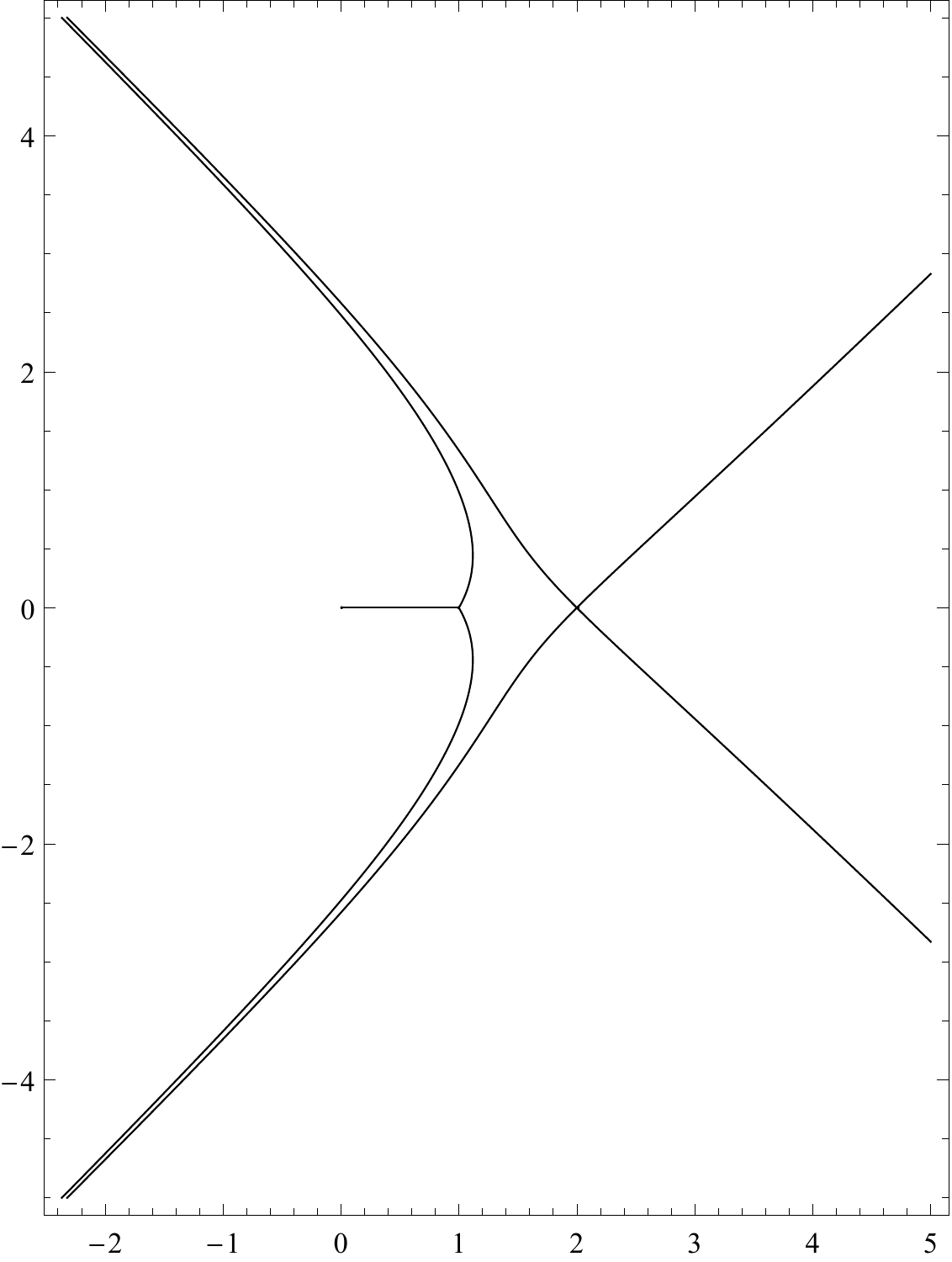}
\end{minipage}
\caption{Critical graphs for $\varpi_{q}$ when $q=\left(  z-1\right)  \left(
z-2\right)  \left(  z-3\right)  $ (left) and $q=\left(  z-1\right)  \left(
z-2\right)  ^{2}$ (right).}%
\label{FIG6}%
\end{figure}\bigskip\bigskip

\section{Connection with the algebraic equation \label{connection alg}}

\bigskip The Cauchy transform $\mathcal{C}_{\nu}$ of a compactly supported
Borelian complex measure $\nu$ is defined in $%
%TCIMACRO{\U{2102} }%
%BeginExpansion
\mathbb{C}
%EndExpansion
\setminus$\emph{supp}$\left(  \nu\right)  $ by :
\[
\mathcal{C}_{\nu}\left(  z\right)  =\int_{%
%TCIMACRO{\U{2102} }%
%BeginExpansion
\mathbb{C}
%EndExpansion
}\frac{d\nu\left(  t\right)  }{z-t}.
\]
It satisfies
\[
\mathcal{C}_{\nu}\left(  z\right)  =\frac{\nu\left(
%TCIMACRO{\U{2102} }%
%BeginExpansion
\mathbb{C}
%EndExpansion
\right)  }{z}+\mathcal{\allowbreak O}\left(  z^{-2}\right)  ,z\rightarrow
\infty,
\]
and the inversion formula (which should be understood in the distributions
sense) :%
\[
\nu=\frac{1}{\pi}\frac{\partial\mathcal{C}_{\nu}}{\partial\overline{z}}.\text{
}%
\]

In particular, the normalized root-counting measure $\nu_{n}=\nu(P_{n})$ of a
complex polynomial $P_{n}$ of a complex polynomial $P_{n}$ of degree $n$ is
defined in $%
%TCIMACRO{\U{2102} }%
%BeginExpansion
\mathbb{C}
%EndExpansion
$ by :
\[
\nu_{n}=\frac{_{1}}{n}\sum\limits_{P_{n}\left(  a\right)  =0}\delta_{a},\text{
(each zero is counted with its multiplicity);}%
\]
the Cauchy transform of $\nu_{n}$ is :
\[
\mathcal{C}_{\nu_{n}}(z)=\int_{%
%TCIMACRO{\U{2102} }%
%BeginExpansion
\mathbb{C}
%EndExpansion
}\frac{d\nu_{n}\left(  t\right)  }{z-t}=\frac{P_{n}^{^{\prime}}(z)}{nP_{n}%
(z)};P_{n}(z)\neq0.
\]

Let us come back to the algebraic equation (\ref{algeq}). We are seeking a
compactly-supported signed measure $\nu$ such that, its Cauchy transform
$\mathcal{C}_{\nu}$ satisfies almost everywhere in $%
%TCIMACRO{\U{2102} }%
%BeginExpansion
\mathbb{C}
%EndExpansion
$ equation (\ref{algeq}).

With the choice of the square root of the discriminant%

\[
\Delta\left(  z\right)  =\frac{z}{4}\left(  4z^{3}+4\gamma z^{2}+\left(
\gamma^{2}-16\right)  z-16\delta\right)
\]
of the quadratic equation (\ref{algeq}) (as a quadratic equation) with
condition
\[
\sqrt{\Delta\left(  z\right)  }\sim z^{2},z\rightarrow\infty,
\]
it is easy to check that independently of the complex numbers $\gamma$ and
$\delta$, we have :
\[
\mathcal{C}(z)=\frac{2z^{2}+\gamma z-2\sqrt{\Delta\left(  z\right)  }}%
{4z}=\frac{1}{z}+\mathcal{\allowbreak O}\left(  z^{-2}\right)  ,z\rightarrow
\infty,
\]
which let us be hopeful for the existence of the measure $\nu.$

The following Lemma gives a sufficient condition on a solution of
(\ref{algeq}) to be the Cauchy transform of some compactly supported measure
in $%
%TCIMACRO{\U{2102} }%
%BeginExpansion
\mathbb{C}
%EndExpansion
$ :

\begin{lemma}
[{\cite[comp. Th. 1.2, Ch. II,]{garnet}}]Suppose $f\in L_{loc}^{1}\left(
%TCIMACRO{\U{2102} }%
%BeginExpansion
\mathbb{C}
%EndExpansion
\right)  $ and that $f(z)\rightarrow0$ as $z\rightarrow\infty$ and let $\mu$
be a compactly-supported measure in $%
%TCIMACRO{\U{2102} }%
%BeginExpansion
\mathbb{C}
%EndExpansion
$ such that%
\[
\mu=\frac{1}{\pi}\frac{\partial f}{\partial\overline{z}}%
\]
in the sense of distributions. Then $f(z)=\mathcal{C}_{\mu}\left(  z\right)  $
almost everywhere in $%
%TCIMACRO{\U{2102} }%
%BeginExpansion
\mathbb{C}
%EndExpansion
.$
\end{lemma}

The following Proposition gives a necessary condition the existence of
measures $\nu$ :

\begin{proposition}
\label{connection}Let us consider the quadratic differential
\begin{equation}
-\frac{\Delta\left(  z\right)  }{z^{2}}dz^{2}. \label{qd2}%
\end{equation}
If the signed measure $\nu$ exists, then, the quadratic differential
(\ref{qd2}) has two short trajectories, and, the support of $\nu$ coincides
with these short trajectories. In particular, if $\Delta\left(  z\right)  $ is
a real polynomial, then the problem of finding the measure $\nu$ is solved.
\end{proposition}

\section{Proofs}

\begin{proof}
[Proof of Lemma \ref{at infinity}]Suppose that $\gamma_{1}$ and $\gamma_{2}$
are two such trajectories emanating from the zero $a$ or $1$, spacing with
angle $\theta\in\left\{  2\pi/3,4\pi/3\right\}  .$ Consider the $\varpi_{q}%
$-polygon with edges $\gamma_{1}$ and $\gamma_{2},$ and vertices $z_{j},$ and
infinity. The right side of (\ref{Teich equality}) can take only the values
$0$ or $-1,$ while the left side is at list $2;$ a contradiction. (Observe the
$\varpi_{qp}$-polygon cannot contain the pole $z=0,$ otherwise it contains
$z=1$ and, again we get a contradiction with (\ref{Teich equality}%
).\bigskip\bigskip
\end{proof}

\begin{proof}
[Proof of Lemma \ref{residue}]Since $\frac{q\left(  t\right)  }{t}$ is a real
rational fraction, then
\begin{equation}
\overline{\sqrt{\frac{q\left(  t\right)  }{t}}}=\sqrt{\frac{q\left(
\overline{t}\right)  }{\overline{t}}},t\neq0, \label{ration symm}%
\end{equation}
and we get, after the change of variable $u=\overline{t}$ in second integral
:
\begin{align*}
\Re\left(  \int_{\overline{z}}^{z}\sqrt{\frac{q\left(  t\right)  }{t}%
}dt\right)   &  =\Re\left(  \int_{1}^{z}\sqrt{\frac{q\left(  t\right)  }{t}%
}dt-\int_{1}^{\overline{z}}\sqrt{\frac{q\left(  t\right)  }{t}}dt\right) \\
&  =\Re\left(  \int_{1}^{z}\sqrt{\frac{q\left(  t\right)  }{t}}dt-\overline
{\int_{1}^{z}\sqrt{\frac{q\left(  t\right)  }{t}}dt}\right) \\
&  =\Re\left(  2i\Im\left(  \int_{1}^{z}\sqrt{\frac{q\left(  t\right)  }{t}%
}dt\right)  \right) \\
&  =0.
\end{align*}
Let us give a necessary condition to get two short trajectories joining two
different pairs of finite critical points of $\varpi_{q}$ in the general case
:
\[
\frac{q\left(  z\right)  }{z}=\frac{z^{3}+\alpha z^{2}+\beta z+\gamma}%
{z}=\frac{\left(  z-a\right)  \left(  z-b\right)  \left(  z-c\right)  }%
{z},a,b,c\in%
%TCIMACRO{\U{2102} }%
%BeginExpansion
\mathbb{C}
%EndExpansion
.
\]
Consider two disjoint oriented Jordan arcs $\gamma_{1}$ and $\gamma_{2}$
connecting two distinct pairs of zeros. We define the single-valued function
$\sqrt{\frac{q\left(  z\right)  }{z}}$ in $%
%TCIMACRO{\U{2102} }%
%BeginExpansion
\mathbb{C}
%EndExpansion
\setminus\left(  \gamma_{1}\cup\gamma_{2}\right)  $ with condition
$\sqrt{\frac{q\left(  z\right)  }{z}}\sim z,z\rightarrow\infty.$ For
$s\in\gamma_{1}\cup\gamma_{2},$ we denote by $\left(  \sqrt{p\left(  s\right)
}\right)  _{+}$ and $\left(  \sqrt{p\left(  s\right)  }\right)  _{-}$the
limits from the $+$ and $-$ sides, respectively. (As usual, the $+$ side of an
oriented curve lies to the left and the $-$ side lies to the right, if one
traverses the curve according to its orientation.)

From the Laurent expansion at $\infty$ of $\sqrt{q\left(  z\right)  }:$%

\[
\sqrt{\frac{q\left(  z\right)  }{z}}=\allowbreak z+\frac{\alpha}{2}-\left(
\frac{\alpha^{2}-4\beta}{8z}\right)  +\allowbreak\mathcal{O}\left(
z^{-2}\right)  ,
\]
we deduce the residue%
\[
res_{\infty}\left(  \sqrt{\frac{q\left(  z\right)  }{z}}\right)  =\frac{1}%
{8}\allowbreak\left(  \alpha^{2}-4\beta\right)  .
\]
Let%
\[
I=\int_{\gamma_{1}}\left(  \sqrt{\frac{q\left(  s\right)  }{s}}\right)
_{+}ds+\int_{\gamma_{2}}\left(  \sqrt{\frac{q\left(  s\right)  }{s}}\right)
_{+}ds.
\]
Since
\[
\left(  \sqrt{\frac{q\left(  s\right)  }{s}}\right)  _{+}=-\left(  \sqrt
{\frac{q\left(  s\right)  }{s}}\right)  _{-},s\in\gamma_{1}\cup\gamma_{2},
\]
we have
\[
2I=\int_{\gamma_{1}\cup\gamma_{2}}\left[  \left(  \sqrt{\frac{q\left(
s\right)  }{s}}\right)  _{+}-\left(  \sqrt{\frac{q\left(  s\right)  }{s}%
}\right)  _{-}\right]  ds=\oint_{\Gamma_{q}}\sqrt{\frac{q\left(  z\right)
}{z}}dz,
\]
where $\Gamma_{q}$ is a closed contours encircling the curves $\gamma_{1}$ and
$\gamma_{2}$. After the contour deformation, we pick up the residue at
$z=\infty,$ and we get
\begin{align*}
I &  =\frac{1}{2}\oint_{\Gamma_{q}}\sqrt{\frac{q\left(  z\right)  }{z}}dz=\pm
i\pi res_{\infty}\left(  \sqrt{p\left(  z\right)  }\right)  \\
&  =\pm\frac{\pi i}{8}\allowbreak\allowbreak\left(  \alpha^{2}-4\beta\right)
\end{align*}
and the necessary condition is
\[
\Im\left(  \allowbreak\alpha^{2}-4\beta\right)  =0,
\]
which is satisfied for the case when $q$ is real.
\end{proof}

\begin{proof}
[Proof of Lemma \ref{curve}]It is clear that $\Sigma\cap%
%TCIMACRO{\U{211d} }%
%BeginExpansion
\mathbb{R}
%EndExpansion
=\left[  0,1\right]  .$ The fact that $\Sigma$ is symmetric with respect to
the real axis follows from the observation (\ref{ration symm}). In order to
prove that $\Sigma$ is a curve, we consider the real functions $F$ and $G$
defined for $\left(  x,y\right)  $ in $%
%TCIMACRO{\U{2102} }%
%BeginExpansion
\mathbb{C}
%EndExpansion
_{+}$ by:
\begin{align*}
F\left(  x,y\right)   &  =\Re\left(  \int_{0}^{x}\sqrt{\frac{\left(  u-\left(
x+iy\right)  \right)  \left(  u-\left(  x-iy\right)  \right)  \left(
u-1\right)  }{u}}du\right) \\
&  =\Re\left(  \int_{0}^{x}\sqrt{\frac{\left(  \left(  u-x\right)  ^{2}%
+y^{2}\right)  \left(  u-1\right)  }{u}}du\right)  ;\\
G\left(  x,y\right)   &  =\Re\left(  \int_{x}^{x+iy}\sqrt{\frac{\left(
u-\left(  x+iy\right)  \right)  \left(  u-\left(  x-iy\right)  \right)
\left(  u-1\right)  }{u}}du\right) \\
&  =-\int_{0}^{1}y^{2}\sqrt{1-t^{2}}\Im\sqrt{1-\frac{1}{x+ity}}dt.
\end{align*}
Observe that
\[
\Sigma=\left\{  \left(  x,y\right)  \in%
%TCIMACRO{\U{211d} }%
%BeginExpansion
\mathbb{R}
%EndExpansion
^{2}\mid\left(  F+G\right)  \left(  x,y\right)  =0\right\}  .
\]
We prove first that $\Sigma\setminus\left[  0,1\right]  \subset\left\{  z\in%
%TCIMACRO{\U{2102} }%
%BeginExpansion
\mathbb{C}
%EndExpansion
\mid\Re z>1\right\}  .$ If $x\leq1$ and $y>0,$ then, it is obvious that,
$F\left(  x,y\right)  =0.$ By the other hand, we have for $0<t\leq1$ :%
\begin{align}
0  &  <\arg\left(  x+ity\right)  <\arg\left(  x-1+ity\right)  <\pi
\label{eq arg}\\
&  \Longrightarrow0<\arg\left(  1-\frac{1}{x+ity}\right)  <\pi\Longrightarrow
\arg\sqrt{1-\frac{1}{x+ity}}\in\left]  0,\frac{\pi}{2}\right[ \nonumber\\
&  \Longrightarrow\Im\sqrt{1-\frac{1}{x+ity}}>0\Longrightarrow G\left(
x,y\right)  <0.\nonumber
\end{align}
Hence, $\left(  F+G\right)  \left(  x,y\right)  <0$ which proves that $\left(
x,y\right)  \notin\Sigma.$

Let us prove now that $\Sigma$ is a curve in the set
\[
\Pi=\left\{  \left(  x,y\right)  ;x>1,y>0\right\}  .
\]
We have
\[
\frac{\partial F}{\partial x}\left(  x,y\right)  =\sqrt{\frac{y^{2}\left(
x-1\right)  }{x}}+\int_{1}^{x}\frac{\left(  x-u\right)  \left(  u-1\right)
}{\sqrt{\left(  \left(  u-x\right)  ^{2}+y^{2}\right)  \left(  u-1\right)  u}%
}dt>0.
\]
By the other hand, with $u_{t}=x+ity,$ $t\in\left[  0,1\right]  ,$ we get
\begin{align*}
\frac{\partial G}{\partial x}\left(  x,y\right)   &  =\frac{\partial}{\partial
x}\left[  \Re\left(  \int_{0}^{1}iy^{2}\sqrt{1-t^{2}}\sqrt{1-\frac{1}{u_{t}}%
}dt\right)  \right] \\
&  =-\int_{0}^{1}\frac{y^{2}\sqrt{1-t^{2}}}{2}\Im\left(  \frac{1}{u_{t}%
^{2}\sqrt{1-\frac{1}{u_{t}}}}\right)  dt
\end{align*}
It is sufficient to prove that,
\[
\forall t\in\left[  0,1\right]  ,\Im\left(  \frac{1}{u_{t}^{2}\sqrt{1-\frac
{1}{u_{t}}}}\right)  \leq0,
\]
which is equivalent to prove that,%
\[
\forall t\in\left[  0,1\right]  ,\arg\left(  \frac{1}{u_{t}^{2}\sqrt
{1-\frac{1}{u_{t}}}}\right)  \in\left[  \pi,2\pi\right[  ,
\]
where the argument is taken in $\left[  0,2\pi\right[  $. It follows from
(\ref{eq arg}) that for any $t\in\left[  0,1\right]  $ :%
\[
\arg\left(  \frac{1}{u_{t}^{2}\sqrt{1-\frac{1}{u_{t}}}}\right)  =2\pi-\left(
\frac{3}{2}\arg\left(  u_{t}\right)  +\frac{1}{2}\arg\left(  u_{t}-1\right)
\right)  \in\left]  \pi,2\pi\right[  .
\]
We deduce that for any $0\leq t\leq1,$ $\Im\left(  \frac{1}{u_{t}^{2}%
\sqrt{1-\frac{1}{u_{t}}}}\right)  \leq0,$ and then $\frac{\partial G}{\partial
x}\left(  x,y\right)  \geq0.$ Finally, we just proved that
\[
\frac{\partial\left(  F+G\right)  }{\partial x}\left(  x,y\right)
\neq0,\left(  x,y\right)  \in\Sigma\cap\Pi.
\]
We conclude that the set $\Sigma$ is a curve in $%
%TCIMACRO{\U{2102} }%
%BeginExpansion
\mathbb{C}
%EndExpansion
$ by $\allowbreak$applying the Implicit Function Theorem to the function $F+G$.
\end{proof}

\begin{proof}
[Proof of Lemma \ref{asympt of the curve}]Let us put $z=re^{ix}\in
\Sigma,r>1,x\in\left[  0,\frac{\pi}{2}\right]  .$ With the change of variable
$t=sre^{ix}$, we get%
\[
\Re\left(  e^{2ix}\int_{0}^{1}\sqrt{\frac{\left(  s-\frac{1}{r}e^{-ix}\right)
\left(  s-1\right)  \left(  s-e^{-2ix}\right)  }{s}}ds\right)  =0.
\]
Taking the limits when $r\rightarrow\infty,$ we get%
\begin{equation}
0=\Re\int_{0}^{1}e^{2ix}\sqrt{\left(  s-1\right)  \left(  s-e^{-2ix}\right)
}. \label{0}%
\end{equation}
We see that $x\neq0;$ suppose that $x\neq\frac{\pi}{2}.$ With the change of
variable $t=\alpha u+\beta,$ where%
\[
\beta=\frac{1+e^{-2ix}}{2},\alpha=i\frac{1-e^{-2ix}}{2},
\]
(\ref{0}) gives%
\begin{align*}
0  &  =\Re\left(  \int_{\cot x}^{i}\sqrt{u^{2}+1}du\right)  =\Re\left(
\int_{\cot x}^{0}\sqrt{u^{2}+1}du+\int_{0}^{i}\sqrt{u^{2}+1}du\right) \\
&  =\Re\left(  \int_{\cot x}^{0}\sqrt{u^{2}+1}du\right)  >0;
\end{align*}
a contradiction.

The Laurent serie of $\sqrt{\frac{\left(  t-1\right)  \left(  t-z\right)
\left(  t-\overline{z}\right)  }{t}}$ when $t\rightarrow1$ is :%
\[
\sqrt{\frac{\left(  t-1\right)  \left(  t-z\right)  \left(  t-\overline
{z}\right)  }{t}}=\left\vert z-1\right\vert \allowbreak\sqrt{t-1}+o\left(
\left(  t-1\right)  ^{\frac{1}{2}}\right)  .
\]
We conclude that
\[
0=\lim\limits_{z\rightarrow1,z\in\Sigma^{+}}\Re\int_{1}^{z}\sqrt{\frac{\left(
t-1\right)  \left(  t-z\right)  \left(  t-\overline{z}\right)  }{t}}%
dt=\frac{2}{3}\left\vert z-1\right\vert \Re\left(  z-1\right)  ^{\frac{3}{2}%
},
\]
and then%
\[
\arg\left(  z-1\right)  ^{\frac{3}{2}}\equiv\frac{\pi}{2}\operatorname{mod}%
\left(  \pi\right)  ,
\]
which finishes the proof.
\end{proof}

\begin{proof}
[Proof of Proposition \ref{main}]It is clear that the segment $\left[
0,1\right]  $ is always a short trajectory of $\varpi_{q}$. If $a\notin%
\Sigma,$ then, from (\ref{cond necess})there is no short trajectory connecting
$a$ to $0$ or $1.$ From Lemma \ref{at infinity}, at most two critical
trajectories emanating from $a$ can diverge to $\infty$ in the upper
half-plane $%
%TCIMACRO{\U{2102} }%
%BeginExpansion
\mathbb{C}
%EndExpansion
^{+}.$ By consideration of symmetry with respect to the real axis, at list one
critical trajectory emanating from $a$ meets a critical trajectory emanating
from $\overline{a},$ at some point $b\in%
%TCIMACRO{\U{211d} }%
%BeginExpansion
\mathbb{R}
%EndExpansion
\setminus\left[  0,1\right]  .$ Since $b$ cannot be a zero of the quadratic
differential $\varpi_{q},$ we conclude that these two critical trajectories
form a short one.

If $a\in\Sigma$, and no short trajectory connecting $a$ to $z=1,$ then, there
exist two critical trajectories $\gamma_{a}$ and $\gamma_{1}$ emanating
respectively from $a$ and $1$ and diverging to infinity in a same direction
$D_{k}$. \bigskip From the behaviour of orthogonal trajectories at $\infty,$
we can take an orthogonal trajectory $\sigma$ that hits $\gamma_{1}$ and
$\gamma_{a}$ respectively in two points $b$ and $c$ (there are infinitely many
such orthogonal trajectories $\sigma$ ). We consider a path $\gamma$
connecting $z=1$ and $a$ formed by the part of $\gamma_{1}$ from $z=1$ to $b,$
the part of $\sigma$ from $b$ to $c,$ and the part of $\gamma_{a}$ from $c$ to
$a.$ Then
\begin{align*}
\Re\int_{\gamma}\sqrt{p\left(  t\right)  }dt  &  =\Re\int_{1}^{b}%
\sqrt{p\left(  t\right)  }dt+\Re\int_{b}^{c}\sqrt{p\left(  t\right)  }%
dt+\Re\int_{c}^{a}\sqrt{p\left(  t\right)  }dt\\
&  =\Re\int_{b}^{c}\sqrt{p\left(  t\right)  }dt\neq0,
\end{align*}
which contradicts the fact $a\in\Sigma.$
\end{proof}

\begin{proof}
[Proof of Proposition \ref{connection}]The fact that the support of $\nu$ is
formed by horizontal trajectories of the quadratic differential (\ref{qd2}) is
classic and it is based on the so-called Plemelj-Sokhotsky Formula. For more
details, we refer the reader to \cite{amf rakh1},\cite{pritsker}%
,\cite{Shapiro},\cite{bullgard}...

Since the Cauchy transform is a single-valued function in $%
%TCIMACRO{\U{2102} }%
%BeginExpansion
\mathbb{C}
%EndExpansion
\setminus$\emph{supp}$\left(  \nu\right)  ,$ then, the support of $\nu$ should
include all the singular points (finite critical points) of the quadratic
differential (\ref{qd2}). But, $\allowbreak$the horizontal trajectories that
contain all finite critical points are exactly short trajectories. The measure
$\nu$ is absolutely continuous with respect to the linear Lebesgue measure,
and it is given on its support (with an adequate orientation) by the
expression :
\[
d\nu\left(  t\right)  =\frac{1}{8i\pi}\left(  \sqrt{\frac{\Delta\left(
t\right)  }{t}}\right)  _{+}dt.
\]
It is easy to check that the Cauchy transform of $\nu$ satisfies
(\ref{algeq}), indeed :
\begin{align*}
C_{\nu}\left(  z\right)   &  =\frac{1}{2i\pi}\int\frac{\left(  \sqrt
{\frac{\Delta\left(  t\right)  }{t}}\right)  _{+}}{4\left(  z-t\right)
}dt=\frac{1}{4i\pi}\oint\frac{\sqrt{\frac{\Delta\left(  t\right)  }{t}}%
}{4\left(  z-t\right)  }dt\\
&  =\frac{1}{4}\left(  res_{z}\frac{\sqrt{\frac{\Delta\left(  t\right)  }{t}}%
}{z-t}-res_{\infty}\frac{\sqrt{\frac{\Delta\left(  t\right)  }{t}}}%
{z-t}\right) \\
&  =\frac{-\sqrt{\frac{\Delta\left(  z\right)  }{z}}+\gamma+2z}{4}%
=\frac{-\sqrt{z\Delta\left(  z\right)  }+\gamma z+2z^{2}}{4z},
\end{align*}
where the path of integration in the first integral is formed by the two short
trajectories, and, in the second integral is a closed contour including the
two short trajectories and far away from $z.$
\end{proof}

\begin{acknowledgement}
\bigskip\ This work was partially supported by the laboratory research
"Mathematics and Applications", Faculty of sciences of Gab\`{e}s. Tunisia.
\end{acknowledgement}

\texttt{Mohamed Jalel Atia(jalel.atia@gmail.com)}

\texttt{Wafaa Karrou(Foufa40@hotmail.com)}

\texttt{Mondher Chouikhi(chouikhi.mondher@gmail.com)}

\texttt{Faouzi Thabet(faouzithabet@yahoo.fr)}

\texttt{Department of mathematics, Facult\'{e} des sciences de Gab\`{e}s, }

\texttt{Cit\'{e} Erriadh Zrig 6072}. \texttt{Gab\`{e}s. Tunisia.}

\bigskip
\end{document}